\documentclass[letterpaper, reqno, 11pt]{amsart}
\usepackage[english]{babel}
\usepackage{graphicx,color}
\usepackage[all]{xy}
\usepackage{amsfonts,dsfont,mathrsfs}
\usepackage{enumerate}
\usepackage{amsmath,amscd}
\usepackage{amsmath}
\usepackage{amssymb}
\usepackage{amsthm}
\usepackage{tikz} 
\usepackage{tikz-cd}
\usepackage{graphics}
\usepackage{listings}
\usepackage{xcolor}
\usepackage{gensymb}
\usepackage{amssymb}
\usepackage{setspace}
\usepackage{fancyhdr}
\usepackage{dsfont}
\usepackage{bm}
\usepackage{enumitem}

\numberwithin{equation}{section}

\newtheorem{thm}{Theorem}

\newtheorem{prop}{Proposition}

\newtheorem{lem}{Lemma}

\newtheorem{cor}{Corollary}

\theoremstyle{definition}
\newtheorem{dfn}{Definition}

\newtheorem*{rmk}{Remark}

\newtheorem{conj}{Conjecture}

\thanks{}

\begin{document}

\title{Connection Blocking In Quotients of \textit{Sol} }
\author[R. Bidar]{Mohammadreza Bidar}
\address{Department of Mathematics\\
                 Michigan State University\\
                 East Lansing, MI 48824}

\date{\today}

\begin{abstract}

Let $G$ be a connected Lie group and $\Gamma \subset G$ a lattice. Connection curves of the homogeneous space $M=G/\Gamma$ are the orbits of one parameter subgroups of $G$. To \textit{block} a pair of points $m_1,m_2 \in M$ is to find a \textit{finite} set $B \subset M\setminus \{m_1, m_2 \}$ such that every connecting curve joining $m_1$ and $m_2$ intersects $B$. The homogeneous space $M$ is \textit{blockable} if every pair of points in $M$ can be blocked, otherwise we call it \textit{non-blockable}.
\par
\textit{Sol} is an important Lie group and one of the eight homogeneous Thurston 3-geometries. It is a unimodular solvable Lie group diffeomorphic to $\mathds{R}^3$, and together with the left invariant metric $ds^2=e^{-2z}dx^2+e^{2z}dy^2+dz^2$ includes copies of the hyperbolic plane, which makes studying its geometrical properties more interesting. In this paper we prove that all quotients of $Sol$ are non-blockable. In particular, we show that for any lattice $\Gamma \subset Sol$, the set of non-blockable pairs is a dense subset of $Sol/\Gamma \times Sol/\Gamma$. 

\end{abstract}

\maketitle


\section{Introduction}

Finite blocking is an interesting concept originating as a problem in billiard dynamics and later in the context of Riemannian manifolds. Let $(M,g)$ be a complete connected, infinitely differentiable Riemannian manifold. For a pair of (not necessarily distinct) points $m_1,m_2 \in M$ let $\Gamma(m_1,m_2)$ be the set of geodesic segments joining these points. A set $B \subset M\setminus \{m_1,m_2\}$ is \textit{blocking} if every $\gamma \in \Gamma(m_1,m_2)$ intersects $B$. The pair $m_1,m_2$ is secure if there is a \textit{finite blocking} set $B=B(m_1,m_2)$. A manifold is secure if all pairs of points are secure. If there is a uniform bound on the cardinalities of blocking sets, the manifold is \textit{uniformly secure} and the best possible bound is the \textit{blocking number}. 
\vspace{0.1in}
\par
Now, the first question naturally arising is what Riemannian manifolds are secure. If we focus on closed Riemannian manifolds, there is the following conjecture \cite{Growth of the number of geodesics,Blocking light}:

\begin{conj} A closed Riemannian manifold is secure if and only if it is flat.
\end{conj}

Flat manifolds are uniformly secure, and the blocking number depends only on their dimension \cite{Connecting geodesics and security, Blocking of billiard}. They are also \textit{midpoint secure}, i.e., the midpoints of connecting geodesics yield a finite blocking set for any pair of points \cite{Connecting geodesics and security, Insecurity for compact,Blocking of billiard}. Conjecture 1. says that flat manifolds are the only secure manifolds. This has been verified for several special cases: A manifold without conjugate points is uniformly secure if and only if it is flat \cite{Growth of the number of geodesics,Blocking light}; a compact locally symmetric space is secure if and only if it is flat \cite{Connecting geodesics and security}; the generic manifold is insecure \cite{dense G-delta,Real Analytic metrics,Insecurity is generic}; Conjecture 1. holds for compact Riemannian surfaces with genus bigger or equal than 1 \cite{Insecurity for compact}; any Riemannian metric has an arbitrarily close, insecure metric in the same conformal class \cite{Insecurity is generic}.  
\vspace{0.1in}
\par
Gutkin \cite{Connection blocking} initiated the study of blocking properties of homogeneous spaces. Here, connection curves are the orbits of one-parameter subgroups of $G$. In this context, he speaks of \textit{finite blocking} instead of security; the counterpart of "secure" in this context is the term \textit{connection blockable}, or simply \textit{blockable}. A counterpart of Conjecture 1 for homogeneous spaces is as follows:

\begin{conj} 
Let $M=G/\Gamma$ where where $G$ is a connected Lie group and $\Gamma \subset G$ is a lattice. Then $M$ is blockable if and only if $G=\mathds{R}^n$, i.e., $M$ is a torus.
\end{conj}

Gutkin in \cite{Connection blocking} establishes Conjecture 2 for nilmanifolds. He then proves the homogeneous space $\textrm{SL}(n,\mathds{R})/\textrm{SL}(n,\mathds{Z})$ is not \textit{midpoint blockable}. If $\Gamma$ is a lattice commensurable to $\textrm{SL}(n,\mathds{Z})$, the homogeneous spaces $\textrm{SL}(n,\mathds{R})/\Gamma$ is non-blockable \cite{Connection Blocking SLnR)}. To complete the proof of the conjecture for all solvable Lie group quotients, a good step is to start with $Sol$. We do this by proving the following theorem:

\begin{thm}\label{thm-Mn}
All quotients of $Sol$ are non-blockable. In particular, for every lattice $\Gamma$ in $Sol$, the set of non-blockable pairs is a dense subset of $Sol/\Gamma \times Sol/\Gamma$.
\end{thm}

To prove this theorem we start with a more specific class of lattices in $Sol$; those that are isomorphic to $\mathds{Z}^2 \rtimes_A \mathds{Z}$. By $\mathds{Z}^2 \rtimes_A \mathds{Z}$ we mean the semidirect product where $A\in \textrm{SL}_2(\mathds{Z})$, and $r\in \mathds{Z}$ acts on $\mathds{Z}^2$ as $A^r$ so as the multiplication is given by $(p_1,q_1,r_1)(p_2,q_2,r_2)=((p_1,q_1)+A^{r_1}(p_2,q_2),r_1+r_2)$. 
If $P$ is the eigenvector matrix of $A$, by Proposition \ref{Z2semidpZlattice} the mapping $(p,q,r) \mapsto (P^{-1}(p,q),sr)$ embeds $\mathds{Z}^2 \rtimes_A \mathds{Z}$ into $Sol$ and the image is a lattice. We then solve the blocking problem for some of these lattices. In particular we prove:

\begin{thm}\label{thm-main}
Let $s\neq 0$ be a real number, $A \in \textrm{SL}_2(\mathds{Z})$ be a matrix with eigenvalues $\lambda=\pm e^s, \lambda^{-1}$, and $P \in \textrm{GL}_2(\mathds{R})$ be such that  $PAP^{-1}= \begin{pmatrix} \lambda & 0 \\
0 & \lambda^{-1}  \\
\end{pmatrix}$ with $P_{11}=P_{22}=1$. Let $\Gamma=\Gamma(A)=\{(P(p,q),sr)|p,q,r \in \mathds{Z}\}$ be the corresponding lattice in $Sol$. A pair of points $m_1=g_1\Gamma, m_2=g_2\Gamma \in Sol/\Gamma$ is not  blockable if $g_1^{-1}g_2$ lies on the planes $x=0, y\neq 0$  or 
$y=0,x\neq 0$.
\end{thm}

\begin{rmk}
The above theorem basically shows that if two points are on the planes $x=c$, or $y=c$, (not having the same $y$, or $x$, respectively) then their corresponding cosets in the quotient space are not blockable. Since these planes are isometric to the hyperbolic plane $\mathds{H}$, this theorem reflects an interesting problem in the hyperbolic plane. 
The left invariant metric of $Sol$ is $ds^2=e^{-2z}dx^2+e^{2z}dy^2+dz^2$. The isometries are given by $(c,y,t) \mapsto (y,e^{-t})$ and $(x,c,t) \mapsto (x,e^{t})$ \cite{Lattices in SOl}. This metric is not bi-invariant and the image of one parameter subgroups are not necessarily geodesics in the hyperbolic plane.They are simply lines passing through $(1,0)$. 
\end{rmk}

Interestingly, all lattices of $Sol$ are isomorphic to the semidirect product lattices presented in Theorem \ref{thm-main}, through which we can prove non-blockability of all quotients of $Sol$.
\vspace{0.1in}
\par
The organization of the paper is as follows. In Section 2, we briefly review connection blocking concept for homogeneous spaces; then we formulate one parameter subgroups in $Sol$. In Section 3, we introduce semi-direct product lattices in $Sol$, we also prove all lattices of this class are quasi-isometric. We then present a group presentation for all lattices in $Sol$ according to Moln\'ar \cite{Classification of Sol}; and finally we prove all lattices in $Sol$ are conjugate to semi-direct product lattices. In Section 4, we first prove a few technical lemmas, then proceed to prove the main theorems of the paper. 




\section{Preliminaries}

In this section we review general preliminaries of blocking properties in homogeneous spaces, and the tools needed to state and prove our main result. We follow the notation and discussion in \cite{Connection blocking}.

\subsection{Connection blocking in homogeneous spaces}
Let $G$ be a connected Lie group, $M=G/\Gamma$, $\Gamma \subset G$ a lattice. For $g \in G, m \in M$, $g \cdot m$ denotes the action of $G$ on $M$. Let $\mathfrak{G}$ be the Lie Algebra of $G$ and let $\exp : \mathfrak{G} \rightarrow G$ be the exponential map. For $m_1,m_2 \in M$ let $C_{m_1,m_2}$ be the set of parametrized curves $c(t)=\exp(tx)\cdot m, 0 \leq t \leq 1$, such that $c(0)=m_1, c(1)=m_2$. We say that $C_{m_1,m_2}$ is the collection of \textit{connecting curves} for the pair $m_1,m_2$. Let $I \subset \mathds{R}$ be any interval. If $c(t), t\in I$, is a curve, we denote by $c(I)\subset M$ the set $\{c(t): t\in I\}$. A \textit{finite set} $B \subset M \ \{m_1,m_2 \}$ is a \textit{blocking set} for the pair $m_1,m_2$ if for any curve $c$ in $C_{m_1,m_2}$ we have $c([0,1])\cup B \neq \emptyset$.
If a blocking set exists, the pair $m_1,m_2$ is \textit{connection blockable}, or simply \textit{blockable}. 
The analogy with Riemannian security \cite{Blocking of billiard,Blocking light,Blocking: new examples,Insecurity for compact: commentary} suggests the following:

\begin{dfn} Let $M=G/\Gamma$ be a homogeneous space.  
\end{dfn}
\begin{enumerate}[label=\roman*)]
\item  $M$ is \textit{connection blockable} if every pair of its points is blockable. If there exists at least one non-blockable pair of points in $M$, then $M$ is non-blockable. 
\item $M$ is \textit{uniformly blockable} if there exists $N \in \mathds{N}$ such that every pair of its points can be blocked with a set $B$ of cardinality at most $N$. The smallest such $N$ is the blocking number for $M$. 
\item A pair $m_1,m_2 \in M$ is \textit{midpoint blockable} if the set $\{ c(1/2) : c \in C_{m_1,m_2} \}$ is finite. A homogeneous space is midpoint blockable if all pairs of its points are midpoints blockable. 
\item A homogeneous space is \textit{totally non-blockable} if no pair of its points is blockable.
\end{enumerate}
\vspace{0.1in}
\par
Consider the homogeneous space $M=G/\Gamma$, where $\Gamma \subset G$ is a lattice.
$M$ carries some straightforward and expected blocking properties. In particular, it is clear from definitions that a pair of points $m_1=g_1\Gamma, m_2=g_2\Gamma$ is non-blockable if and only if $m=g_1^{-1}g_2\Gamma$ is not blockable from the identity $m_0=\Gamma$. Thus, $M$ is blockable (resp. uniformly blockable, midpoint blockable) if and only if all pairs $m_0,m$ are blockable (resp. uniformly blockable,midpoint blockable). The space $M$ is totally non-blockable if and only if no pair $m_0,m$ is blockable. 
\vspace{0.1in}
\par
We say homogeneous spaces $M_1,M_2$ have \textit{identical blocking property} if both are blockable (or not), midpoint blockable (or not), totally non-blockable (or not), etc. 
\vspace{0.1in}
\par
Let $\textrm{exp}:\mathfrak{G} \rightarrow G$ be the exponential map. For $\Gamma \subset G$ denote by $p_{\Gamma}:G \rightarrow G/\Gamma$ the projection, and set $\textrm{exp}_{\Gamma}=p_{\Gamma} \circ \textrm{exp}: \mathfrak{G} \rightarrow G/\Gamma$. We will say that a pair $(G,\Gamma)$ is of \textit{exponential type} if the map $\textrm{exp}_{\Gamma}$ is surjective. Let $M=G/\Gamma$. For $m \in M$ set $\textrm{Log}(m)=\textrm{exp}_{\Gamma}^{-1}(m)$. Note, $\textrm{Log}(m)$ may have more than one element. We will use the following basic fact to prove a point is not blockable from identity. See \cite{Connection blocking} Proposition 2 for the proof.

\begin{prop}\label{prop:fcosets}
Let $G$ be a Lie group, $\Gamma \subset G$ a lattice such that $(G,\Gamma)$ is of exponential type, and let $M=G/\Gamma$. 
Then $m\in M$ is blockable away from $m_0$ if and only if there is a map $x \mapsto t_x$ of $Log(m)$ to $(0,1)$ such that the set $\{\exp(t_x x): x \in Log(m) \}$ is contained in a finite union of $\Gamma$-cosets. 
\end{prop}

\subsection{\textit{Sol} and One parameter subgroups}

In this section we derive an explicit formula for one parameter subgroups in $Sol$, which is  essential to study its blocking properties. 

\begin{dfn}
By $Sol$ we mean the Lie group $\mathds{R}^2\rtimes \mathds{R}$ where $z\in\mathds{R}$
acts on $\mathds{R}^2$ as $\begin{pmatrix} e^z & 0 \\
0 & e^{-z}  \\
\end{pmatrix}$, so as multiplication is given by $(x_1,y_1,z_1)(x_2,y_2,z_2)=(x_1+e^zx_2,y_1+e^{-z}y_2,z_1+z_2)$, together with a left invariant Riemannian metric
$ds^2=e^{-2z}dx^2+e^{2z}dy^2+dz^2$.
\end{dfn}

Consider the three curves $\mathds{R}\rightarrow Sol$ given by $\gamma_1:t \mapsto (t,0,0)$, $\gamma_2:t \mapsto (0,t,0)$ and $\gamma_3:t \mapsto (0,0,t)$. These have tangent vectors $\frac{\partial \gamma_1}{\partial t}=\frac{\partial}{\partial x}$,
$\frac{\partial \gamma_2}{\partial t}=\frac{\partial}{\partial y}$ and 
$\frac{\partial \gamma_3}{\partial t}=\frac{\partial}{\partial z}$ at $(0,0,0)$ respectively, and these vectors span the tangent space at that point. The left action of the group on these vectors gives a collection of three invariant vector fields $X_1, X_2$ and $X_3$ which form a basis for the tangent space at each point. $(x,y,z)\gamma_1 \mapsto (x+e^zt,y,z)$ so $X_1(x,y,z)=\frac{\partial}{\partial t}{ (x,y,z)\gamma_1} \mid_{t=0} = e^z \frac{\partial}{\partial x}$, Similarly, $(x,y,z)\gamma_2 \mapsto (x,y+e^{-z}t,z)$ so $X_2(x,y,z)=\frac{\partial}{\partial t}{ (x,y,z)\gamma_2} \mid_{t=0} = -e^{-z} \frac{\partial}{\partial y}$ and $(x,y,z)\gamma_3 \mapsto (x,y,z+t)$ so $X_3(x,y,z)=\frac{\partial}{\partial t}{ (x,y,z)\gamma_3} \mid_{t=0} = \frac{\partial}{\partial z}$. We construct the metric to be orthogonal at every point with respect to these vector fields. Thus $(\frac{\partial}{\partial x}\mid_{(x,y,z)},\frac{\partial}{\partial x}\mid_{(x,y,z)})=(e^{-z}X_1(x,y,z),e^{-z}X_1(x,y,z))= e^{-2z}$, $(\frac{\partial}{\partial y}\mid_{(x,y,z)},\frac{\partial}{\partial y}\mid_{(x,y,z)})=(-e^{z}X_2(x,y,z),-e^{z}X_2(x,y,z))= e^{2z}$ and $(\frac{\partial}{\partial z}\mid_{(x,y,z)},\frac{\partial}{\partial z}\mid_{(x,y,z)})=(X_3(x,y,z),X_3(x,y,z))= 1$ and so we obtain the metric given above. 
\vspace{0.1in}
\par
Let $\mathfrak{Sol}$ denote the Lie algebra of left invariant vector fields in $Sol$, together with the basis $X_1,X_2,X_3$ as above. We have the following proposition:

\begin{prop}
The exponential map of $Sol$ is given as the following: Given any vector $X=a_1X_1+a_2X_2+a_3X_3 \in \mathfrak{Sol}$, $\exp(tX)=(\frac{a_1}{a_3}(e^{a_3t}-1),\frac{a_2}{a_3}(e^{-a_3t}-1),a_3t)$, if $a_3\neq 0$. If $a_3=0$,
$\exp(tX)=(a_1t,-a_2t,0)$.
\end{prop}

\begin{proof}
Let $\gamma(t)=(x(t),y(t),z(t))$ be the integral curve to $X$ so that $\gamma'(t)=X(t)=a_1e^z\frac{\partial}{\partial x}-a_2e^{-z}\frac{\partial}{\partial y}+a_3\frac{\partial}{\partial z}$. This leads to the first order system $x'(t)=a_1e^{z(t)},\, y'(t)=-a_2e^{-z(t)},\, z'(t)=a_3,\ \gamma(0)=(0,0,0)$ which can be easily solved giving the exponential formula. 
\end{proof}

\begin{rmk}
For every $g\in Sol$, the exponential map formula shows that the equation $\exp(X)=g$ has a unique solution. Let $g^t=\exp(t\log g)$ be the unique one parameter subgroup joining identity and $g$. A direct computation gives us the following corollary. 
\end{rmk}

\begin{cor}\label{g^t for Sol}
$g^t=\left(\frac{x}{e^z-1}(e^{tz}-1),\frac{y}{e^{-z}-1}(e^{-tz}-1),tz\right)$ for $g=(x,y,z) \in Sol,\  z\neq 0$. If $g=(x,y,0)$, $g^t=(tx,-ty,0)$.
\end{cor}

\section{Lattices in $Sol$}
A complete classification of lattices in $Sol$ is presented in \cite{Classification of Sol}. In this paper, $Sol$ lattices are classified in an algorithmic way into 17 different types, but infinitely many $Sol$ affine equavalence classes, in each type. For the purpose of the blocking problem, we consider a class of lattices constructed by the following proposition. We then prove, every lattice in $Sol$ is conjugate to a lattice in this class. 

\begin{prop}\label{Z2semidpZlattice}
Let $A \in \textrm{SL}_2(\mathds{Z})$. Suppose that $A$ is conjugate in $\textrm{GL}_2(\mathds{R})$ to a matrix of the form $\begin{pmatrix} \lambda & 0 \\
0 & \lambda^{-1}  \\
\end{pmatrix}$ for some positive $\lambda \neq 1$. Then there is a quasi-isometric embedding 
$\mathds{Z}^2 \rtimes_A \mathds{Z}\hookrightarrow Sol$ and the image is a lattice. In particular if $A$ and $B$ are both matrices of the above form, then $\mathds{Z}^2 \rtimes_A \mathds{Z}$ is quasi-isometric to $\mathds{Z}^2 \rtimes_B \mathds{Z}$.
\end{prop}

Note that by $\mathds{Z}^2 \rtimes_A \mathds{Z}$ we mean the semidirect product where $r\in \mathds{Z}$ acts on $\mathds{Z}^2$ as $A^r$ so as the multiplication is given by $(p_1,q_1,r_1)(p_2,q_2,r_2)=((p_1,q_1)+A^{r_1}(p_2,q_2),r_1+r_2)$.

\begin{proof}
By assumption there exists $P \in \textrm{GL}_2(\mathds{R})$ such that $PAP^{-1}=\begin{pmatrix} \lambda & 0 \\
0 & \lambda^{-1}  \\
\end{pmatrix}$, and $s \in \mathds{R} \setminus \{0\}$ such that $\begin{pmatrix} \lambda & 0 \\
0 & \lambda^{-1}  \\
\end{pmatrix}=\begin{pmatrix} e^s & 0 \\
0 & e^{-s}  \\
\end{pmatrix}$. Define the embedding by $(p,q,r) \mapsto \left( P(p,q),sr\right)$
and note that since $s \neq 0$, and $P$ is nonsingular this is an injection. The following calculation demonstrate that this gives a homomorphism:
$(p_1,q_1,r_1)(p_2,q_2,r_2)=((p_1,q_1)+A^{r_1}(p_2,q_2),r_1+r_2)\mapsto \left(P(p_1,q_1)+PA^{r_1}(p_2,q_2),s(r_1+r_2) \right)=( P(p_1,q_1)+\begin{pmatrix} e^s & 0 \\
0 & e^{-s}  \\
\end{pmatrix}^{r_1}P(p_2,q_2),sr_1+sr_2)=( P(p_1,q_1)+\begin{pmatrix} e^{sr_1} & 0 \\
0 & e^{-sr_1}  \\
\end{pmatrix}P(p_2,q_2),sr_1+sr_2)=\left(P(p_1,q_1),sr_1\right)\left(P(p_2,q_2),sr_2\right)$ The quotient of $Sol$ by $\mathds{Z}^2 \rtimes_A \mathds{Z}$ is a $\mathds{T}^2$ bundle over $S^1$ so is compact. Thus $\mathds{Z}^2 \rtimes_A \mathds{Z}$ is indeed a lattice in $Sol$. We now show that the action of $\mathds{Z}^2 \rtimes_A \mathds{Z}$ on $Sol$ is proper. Let $g=(X,Y,Z) \in Sol$ and let $\gamma=(p,q,r) \in \mathds{Z}^2 \rtimes_A \mathds{Z}\setminus \{1\}$. Then $\gamma g=(P(p,q)+(e^rX,e^{-r}Y),sr+Z)$.
If $r\neq 0$ then $d(g,\gamma g)\geq |s| \geq 0$. If $r=0$ then $\gamma g=(P(p,q)+(X,Y),Z)$ and both $g$ and $\gamma g$ lie in the same horizontal plane $z=Z$ on which the metric restricts to $ds^2=e^{-2Z}dx^2+e^{2Z}dy^2+dz^2$. In this case let $\mu =\min\{e^{-2Z},e^{2Z}\}>0$ and let $K=\inf_{\| (x,y) \|_{2}=1} \| P(x,y) \|_{2}>0$. Then $d(g,\gamma g) \geq \mu K\| (p,q) \|_{2}$ and since $\gamma \neq 1$, $(p,q)\neq (0,0)$ 
so $d(g, \gamma g)\geq \mu K$. We have thus shown that for all $\gamma \in \mathds{Z}^2 \rtimes_A \mathds{Z}$ with $\gamma \neq 1$, $d(g, \gamma g)\geq \min \{s,\mu K\}>0$. Hence the action of $\mathds{Z}^2 \rtimes_A \mathds{Z}$ on $Sol$ is proper. Since the action is also cocompact the Svarc-Milnor Lemma says that the embedding is a quasi-isometry.
\end{proof}

$Sol$ multiplication can be projectively interpreted by "left translations" on its points as $L_{\tau}: (x,y,z) \mapsto \tau(x,y,z),\, \tau \in Sol$. Let $L(T)$ denote the set of left translations on $Sol$ and assume $\Gamma<L(T)$ is a subgroup, generated by three independent translations $\tau_1=(x_1,y_1,z_1), \tau_2=(x_2,y_2,z_2), \tau_3=(x_3,y_3,z_3)$ with non-commutative addition, or in this case $\mathds{Z}$ linear combinations. The concept of a lattice can be rephrased as a subgroup of left translations. The theorem below clarifies the algebraic structure of lattices in $Sol$ (see \cite{Classification of Sol}):

\begin{thm}\label{alllattices}
Each lattice $\Gamma$ of $Sol$ has a groups presentation 
$$\Gamma=\Gamma(A)=\langle\tau_1,\tau_2,\tau_3: [\tau_1,\tau_2]=1, \tau_3^{-1}\tau_1\tau_3=\tau_1A^P, \tau_3^{-1}\tau_2\tau_3=\tau_2A^P \rangle,$$
where $[\tau_1,\tau_2]$ denotes the commutator $\tau_1^{-1}\tau_2^{-1}\tau_1\tau_2$, and 
$A=\begin{pmatrix} a & b \\
c & d  \\
\end{pmatrix} \in \textrm{SL}_2(\mathds{Z})$ with $\textrm{tr}(A)>2$, $\tau_1=(x_1,y_1,z_1), \tau_2=(x_2,y_2,z_2)$ satisfy the equalities $z_1=z_2=0$, and the matrix 
$P=\begin{pmatrix} x_1 & y_1 \\
x_2 & y_2  \\
\end{pmatrix} \in \textrm{GL}_2(\mathds{R}$ satisfies: $P^{-1}A P:= A^P =\begin{pmatrix} e^{z_3} & 0 \\
0 & e^{-z_3}  \\
\end{pmatrix}$ that is just a hyperbolic rotation fixed by the component $z_3$ in $\tau_3$ above. 
\end{thm}

\begin{rmk}
Moln\'ar definition of $sol$ multiplication is slightly different from our definition. In his paper he defines the lattices as a subgroup of right translations of $Sol$. As a result, the statement of Theorem \ref{alllattices} has been readjusted accordingly. 
\end{rmk}

Using notations of Proposition \ref{Z2semidpZlattice} and Theorem \ref{alllattices}, it's easy to see that lattices of Proposition \ref{Z2semidpZlattice}, correspond to $\Gamma(A)=\langle \tau_1,\tau_2,\tau_3 \rangle$ where, $x_3=y_3=0$. Then $e^{z_3}, e^{-z_3}$ are eigenvalues of $A$ and $P$ is the eigenvector matrix of $A$. Now pairing Proposition \ref{Z2semidpZlattice} and Theorem \ref{alllattices}, we conclude the following proposition which will be used lated to study blocking property of all quotients of $Sol$.

\begin{prop}\label{latticeisomorphic}
Every lattice of $Sol$ is conjugate to a semidirect product lattice presented by Proposition \ref{Z2semidpZlattice}.
\end{prop}

\begin{proof}
Given a lattice $\Gamma=\Gamma(A)=\langle \tau_1,\tau_2,\tau_3 \rangle$ as in Theorem \ref{alllattices}, let $\Gamma_{0}=\Gamma_{0}(A)=\langle \tau_1,\tau_2,\tau'_3=(0,0,z_3) \rangle$, $g=(\frac{x_3}{e^{z_3}-1},\frac{y_3}{e^{-z_3}-1},0)$, $\phi_g \in \textrm{Aut}(Sol):(x,y,z) \mapsto g^{-1} (x,y,z) g$. Since $g$ commutes with $\tau_1,\tau_2$, $\phi_g(\tau_1)=\tau_1, \phi_g(\tau_2)=\tau_2$. A direct computation shows that $\phi_g(\tau'_3)=\tau_3$, hence $\phi_g(\Gamma_0)=\Gamma$.
\end{proof}

\section{Blocking property of $Sol$ quotient spaces}

This section concludes with the proof of the main theorems. We first need a few technical lemmas that will be applied in the body of the proofs. 

\begin{lem}
Let $A=\begin{pmatrix} a & b \\
c & d  \\
\end{pmatrix}\in \textrm{SL}_2(\mathds{Z})$ with eigenvalues $\lambda=e^s\neq 1,\lambda^{-1}$, and $P= \begin{pmatrix} P_{11} & P_{12} \\
P_{21} & P_{22}  \\
\end{pmatrix}\in \textrm{GL}_2(\mathds{R})$ be such that $PAP^{-1}=\begin{pmatrix} \lambda & 0 \\
0 & \lambda^{-1}  \\
\end{pmatrix}$ and $P_{11}=P_{22}=1$. Then $\lambda \notin \mathds{Q}$, $(P(p,q),sr)=(p-\frac{1}{b}(e^s-a)q,q-\frac{1}{c}(e^{-s}-d)p,sr)$. 
\end{lem}

\begin{proof}
Solving the quadratic equation for $\lambda$ it follows that $\lambda=\textrm{tr}(A)/2\pm \sqrt{\textrm{tr}(A)^2-4}/2$, since $\textrm{tr}(A) \in \mathds{Z}$, and $\textrm{tr}(A)=(e^s+e^{-s})/2>2$, $\textrm{tr}(A)^2-4$ can not be a perfect square, so is irrational. Let $\bf{v_1},\bf{v_2}$ be the eigenvectors associated to $\lambda, \lambda^{-1}$, so that the first component of $\bf{v_1}$ and the second component of $\bf{v_2}$ are equal to 1, respectively. Assume $\tilde{P}=[\bf{v_1},\bf{v_2}]$, and let $P=(\tilde{P}/\textrm{det}(\tilde{P}))^{-1}$. A direct computation shows that $P_{11}=P_{22}=1$, also the second statement of the Lemma. 
\end{proof}

Next we prove the following lemma. 

\begin{lem}\label{linear-iso}
Let $A$ be conjugate to its eigenvalue matrix via matrices $P_1,P_2 \in \textrm{GL}_2(\mathds{R})$ as in the statement of Proposition \ref{Z2semidpZlattice}, and $\Gamma_{i},\, i=1,2$ be the two associated lattices in $Sol$, i.e. images of the embeddings $(p,q,r) \mapsto \left( P_{i}(p,q),sr\right)$. Then $B=P_2P_1^{-1}$ is diagonal and the mapping $\phi:Sol \rightarrow Sol$, $(x,y,z) \mapsto (B(x,y),z)$ is a Lie group automorphism. In addition $\phi(\Gamma_1)=\Gamma_2$, hence quotient spaces $Sol/\Gamma_1$ and $Sol/\Gamma_2$ have identical blocking property, i.e. $m=g\Gamma_1$ is blockable from the identity $m^{0}_1=\Gamma$ if and only if $\phi(g)\Gamma'$ is blockable from $m^{0}_2=\Gamma'$. 
\end{lem}

\begin{proof}
Since $P_1^{-1},P_2^{-1}$ are eigenvector matrices of $A$, there exist an invertible diagonal matrix $B_1$ such that $P_2^{-1}=P_1^{-1}B_1$. Letting $B=B_1^{-1}$, it follows that $P_2P_1^{-1}=B$. Since $B$ is nonsingular, $\phi$ is a diffeomorphism on $Sol$, and it's clear from the definition $\phi(\Gamma_1)=\Gamma_2$. The following calculation demonstrates that the mapping is a homomorphism:
$(x_1,y_1,z_1)(x_2,y_2,z_2)=((x_1,y_1)+\begin{pmatrix} e^{z_1} & 0 \\
0 & e^{-z_1}  \\
\end{pmatrix}(x_2,y_2),z_1+z_2) \mapsto (B(x_1,y_1)+B\begin{pmatrix} e^{z_1} & 0 \\
0 & e^{-z_1}  \\
\end{pmatrix}(x_2,y_2),z_1+z_2)=(B(x_1,y_1)+\begin{pmatrix} e^{z_1} & 0 \\
0 & e^{-z_1}  \\
\end{pmatrix}B(x_2,y_2),z_1+z_2)=(B(x_1,y_1),z_1)(B(x_2,y_2),z_2)$.
Since Lie group isomorphisms map one parameter subgroups to one parameter subgroups, the last statement follows immediately.  
\end{proof}

Now we are ready to prove the main theorems. 

\begin{proof}[Proof of Theorem \ref{thm-main}]
Assume that matrix $A=\begin{pmatrix} a & b \\
c & d  \\
\end{pmatrix}\in \textrm{SL}_2(\mathds{Z})$ with eigenvalues $\lambda \neq1,\lambda^{-1}$ is given, and $P \in \textrm{GL}_2(\mathds{R})$ is such that $PAP^{-1}=\begin{pmatrix} \lambda & 0 \\
0 & \lambda^{-1}  \\
\end{pmatrix}$ and $P_{11}=P_{22}=1$ (Note that since switching $A \leftrightarrow -A$ doesn't change $P$ we may assume $\lambda>0$  and since $\lambda \neq 1$, $\textrm{tr}(A)>2$). Let $\lambda=e^s, s\neq 0$, $\Gamma$ be the image lattice of $\mathds{Z}^2 \rtimes_A \mathds{Z}$ through embedding in proposition \ref{Z2semidpZlattice}, $g=(0,y,z) \in Sol$, $y\neq 0$, and $m=g\Gamma$. We prove that $m$ is not blockable from the identity $m_0=\Gamma$. Changing the representative $g$ for $m=g\Gamma$ if necessary, we may assume $z\neq 0$. To the contrary, assume that $m$ is blockable from identity $m_0$. Let $r_i$ be a sequence of integers, so that $sr_i$ is strictly increasing and, $sr_i \rightarrow \infty$, as $i \rightarrow \infty$, and let $\gamma_i=(0,0,s r_i) \in \Gamma$. By proposition \ref{prop:fcosets}, for a suitable choice of $t_{i}$'s where $0<t_{i}<1$, we should have $\{ (g\gamma_{i})^{t_{i}} \} \subset \cup_{n=1}^N \tilde{g_n}\Gamma$; passing to a subsequence if necessary, we may assume $(g\gamma_{i})^{t_{i}} \in \tilde{g}\Gamma$ for some fixed $\tilde{g} \in Sol$, in addition we may assume $\tilde{g}=(g\gamma_{1})^{t_{1}}=(0,\tilde{y},\tilde{z})$. 
\vspace{0.1in}
\par
Since $g\gamma_i=(0,y,z+sr_i)$, By Corollary \ref{g^t for Sol}: 
\begin{equation}\label{ggamma-i-ti}
(g\gamma_{i})^{t_{i}}=\left(0,\frac{y}{e^{-z-sr_i}-1}(e^{-t_i(z+sr_i)}-1),t_i(z+sr_i) \right),
\end{equation}

and thus for all $i,j$ the first component of $[(g\gamma_{i})^{t_{i}}]^{-1}\cdot [(g\gamma_{j})^{t_{j}}]$ is 0. But $[(g\gamma_{i})^{t_{i}}]^{-1}\cdot [(g\gamma_{j})^{t_{j}}] \in \Gamma$ so $[(g\gamma_{i})^{t_{i}}]^{-1}\cdot [(g\gamma_{j})^{t_{j}}]=(\tilde{p}-\frac{1}{b}(e^s-a)\tilde{q},\tilde{q}-\frac{1}{c}(e^{-s}-d)\tilde{p},s\tilde{r})$ for some $\tilde{p},\tilde{q},\tilde{r} \in \mathds{Z}$. So $\tilde{p}-\frac{1}{b}(e^s-a)\tilde{q}=0$ and since $e^{s}=\lambda \notin \mathds{Q}$, we must have $\tilde{p}=\tilde{q}=0$. Therefore it follows that 
\begin{equation}
[(g\gamma_{i})^{t_{i}}]^{-1}\cdot [(g\gamma_{j})^{t_{j}}]=(0,0,s\tilde{r_{ij}}), \ \tilde{r_{ij}} \in \mathds{Z},
\end{equation}
and hence
\begin{equation}\label{s-rij}
t_j(z+sr_j)-t_i(z+sr_i)=s\tilde{r_{ij}}, \ r_{ij} \in \mathds{Z}\, .
\end{equation}
\vspace{0.1in}
\par
Now letting $i=1$, we conclude that 
\begin{equation}\label{ggammaj-ver-line}
(g\gamma_{j})^{t_{j}}=\tilde{g}(0,0,s\tilde{r_{j}})=(0,\tilde{y},\tilde{z}+s\tilde{r_{j}})
\end{equation}
which means $\{(g\gamma_{i})^{t_{i}}\}\cap \tilde{g}\Gamma$ lies on a vertical line in $yz$ plane. 
\vspace{0.1in}
\par
Solving for $t_i$ using the second components of equations \ref{ggamma-i-ti} and \ref{ggammaj-ver-line} it follows that
\begin{equation}
t_i=-\frac{1}{z+sr_i}\ln\left( \frac{\tilde{y}}{y}(e^{-z-sr_i}-1)+1\right);
\end{equation}
Note that $\tilde{y}/y$ has to be positive. Now plugging the formula for $t_i$ and $t_j$ in equation \ref{s-rij} gives us:
\begin{equation}
\frac{e^{-z-sr_i}-1+y/\tilde{y}}{e^{-z-sr_j}-1+y/\tilde{y}}=e^{s\tilde{r_{ij}}}
\end{equation}
Setting, $j=i+1$, $i \rightarrow \infty$, the left side of the above equation goes to 1. So, for large enough $i,j=i+1$, $\tilde{r_{ij}}=0$, and so $r_i=r_{i+1}$ which is a contradiction. 
\vspace{0.1in}
\par
Now, if $g=(x,0,z), x\neq 0$, replace $(0,0,sr_i)$ with $(0,0,-sr_i)$; repeating a similar argument on the first component of $(g\gamma_i)^{t_i}$, proves that $m=(x,0,z)\Gamma$ is also not blockable from $m_0$. 
\end{proof}

Knowing all lattices of $Sol$ are conjugate to semidirect products, we are ready to prove the second theorem. 

\begin{proof}[Proof of Theorem \ref{thm-Mn}]
Given a lattice $\Gamma=\Gamma(A)$ in $Sol$, by Proposition \ref{latticeisomorphic} and Lemma \ref{linear-iso} it is isomorphic to a lattice $\Gamma_0$ presented in Proposition \ref{Z2semidpZlattice} with $P_{11}=P_{22}=1$. From Theorem \ref{thm-main}, all points in $X=\{(0,y,z)\Gamma_0 | y,z \in \mathds{R}, y\neq 0\}$ are not blockable from the identity. We show that $X\Gamma_0$ is dense in $Sol$, which implies $\Gamma_0$ is dense in $Sol/\Gamma_0$. Noting that $X\Gamma_0=\{(0,y,z)\gamma | y,z \in \mathds{R}, y\neq 0, \gamma \in \Gamma_0 \}= \{((0,y)+ \begin{pmatrix} e^z & 0 \\
0 & e^{-z}  \\
\end{pmatrix}P(p,q),z+sr)| y,z \in \mathds{R}, p,q,r \in \mathds{Z}, y\neq 0\}=\{(e^{z}(p-\frac{1}{b}(e^s-a)q),y+e^{-z}(q-\frac{1}{c}(e^{-s}-d)p),z+sr) |
y,z \in \mathds{R}, p,q,r \in \mathds{Z}, y\neq 0\}\, .$ 
\vspace{0.1in}
\par
Since $y\neq 0$ and $z$ are free to vary over $\mathds{R}$, set of the second and third components of $X\Gamma$ are $=\mathds{R}$, in addition since $e^s$ is not rational, for every given $z$, set of the first component of $X\Gamma$ is dense in $\mathds{R}$, and hence $X\Gamma_0$ is dense in $Sol$. Since isomorphic lattices (via an automorphism) of a Lie group carry similar blocking properties, with one to one correspondence between non-blockable pairs, $Sol/\Gamma$ has a dense subset of points, not blockable from identity $m_0=\Gamma$, which implies the statement of Theorem \ref{thm-Mn}. 
\end{proof}

\bibliography{aomsample}
\bibliographystyle{aomalpha}

\end{document}